\documentclass[a4paper,11pt]{article}
\usepackage{makeidx}
\usepackage{graphicx}
\usepackage[english,activeacute]{babel}
\usepackage[latin1]{inputenc}
\usepackage{amsmath}
\usepackage{amsfonts}
\usepackage{amssymb}
\usepackage{latexsym}
\usepackage{colortbl}
\usepackage{multicol}
\usepackage{hyperref}
\usepackage{xcolor}
\usepackage{mathrsfs}

\newtheorem{corollary}{Corollary}

\newtheorem{lemma}{Lemma}

\newtheorem{proposition}{Proposition}

\newenvironment{proof}[1][Proof]{\noindent\textbf{#1.} }{\ \rule{0.5em}{0.5em}}
\begin{document}

\title{The annihilation operator for certain family of $q$-Hermite Sobolev-type orthogonal polynomials}
\author{Carlos Hermoso$^{1}$, Anier Soria--Lorente$^{2}$ \\
\\
$^{1}$Departamento de Física y Matemáticas, Universidad de Alcalá\\
Ctra. Madrid-Barcelona, Km. 33,600\\
28805 - Alcalá de Henares, Madrid, Spain\\
carlos.hermoso@uah.es\\
\\
$^{2}$Departamento de Tecnología, Universidad de Granma\\
Ctra. de Bayamo-Manzanillo, Km. 17,500\\
85100 - Bayamo, Cuba\\
asorial@udg.co.cu, asorial1983@gmail.com}
\date{\emph{\today}}
\maketitle

\begin{abstract}
We present a new family $\left\{ S_{n}(x;q)\right\} _{n\geq 0}$ of monic polynomials in $x$, orthogonal with respect to a Sobolev-type inner product related to the $q$-Hermite I orthogonal polynomials, involving a first-order $q$-derivative on a mass-point $\alpha \in \mathbb{R}$ located out of the corresponding orthogonality interval $[-1,1]$, for some fixed real number $q \in (0, 1)$. We present connection formulas, and the annihilation operator for this non-standard orthogonal polynomial family.


\end{abstract}



\section{Introduction}

\label{S1-Intro}



Let $\left\{ H_{n}(x;q)\right\} _{n\geq 0}$ be the sequence of \textit{monic 
$q$-Hermite I polynomials} of degree $n$, orthogonal with respect to the
inner product%
\begin{equation}
\langle f,g\rangle =\int_{-1}^{1}f(x;q)g(x;q)(qx,-qx;q)_{\infty
}d_{q}x,\;0<q<1,\quad f,g\in \mathbb{P}  \label{qHermiteInnP}
\end{equation}%
in the linear space $\mathbb{P}$ of polynomials of real coefficients. This
family of $q$-hypergeometric polynomials was introduced at the end of the
nineteenth century by L. J. Rogers (see \cite{R1893-1} \cite{R1894-2}, and 
\cite{R1895-3}) and studied throughout the twentieth century by Szeg\H{o}
(see \cite{Sz1926}) and Carlitz (see \cite{C1956} \cite{C1957}, and \cite%
{C1972}), currently playing and important role in different fields such as,
for instance, classical and non-commutative probability theory, combinatory
or quantum-physics. The monic $q$-Hermite I polynomials can be also
described as the family of polynomials satisfying the orthogonality relation%
\begin{equation}
\int_{-1}^{1}H_{m}(x;q)H_{n}(x;q)(qx,-qx;q)_{\infty
}d_{q}x=(1-q)(q;q)_{n}(q,-1,-q;q)_{\infty }q^{\binom{n}{2}}\delta _{m,n},
\label{qHermiteInnP2}
\end{equation}%
where $\delta _{m,n}$ is the well known Kronecker-delta symbol (see \cite{BI-CJM96}). $\left\{
H_{n}(x;q)\right\} _{n\geq 0}$ form a system of polynomials, orthogonal with
respect to the measure $(qx,-qx;q)_{\infty }d_{q}x$, where $q\in \mathbb{R}$
stands for its unique fixed parameter, for which we assume that $0<q<1$.
This last condition means that they belong to the class of orthogonal
polynomial solutions of certain second order $q$-difference equations, known
in the literature as the Hahn class (see \cite{H1949} and \cite{KLS2010}). It
is also important to highlight that, for $q=1$ one recovers the classical
Hermite polynomials and, for $q=0$ one recovers a re-scaled version of the
Chebyshev polynomials of the second kind.

On the other hand, the introduction of derivatives or $q$-derivatives acting on orthogonal polynomials constitutes a powerful tool to obtain new families of orthogonal polynomials. The polynomials resulting under the action of the $q$-derivatives on the $q$-Hermite family, as well as their orthogonality properties, have recently been studied (see \cite{Adig}). Furthermore, $q$-derivatives can also appear involved in the discrete part of the measure, thus modifying the continuous part and giving rise to Sobolev-type perturbations of the $q$-Hermite I polynomials. For a recent and comprehensive study on general discrete Sobolev orthogonal polynomials including difference operators see \cite{Galin}.

In this contribution, we inquired on the following Sobolev-type modification%
\begin{equation}
\left\langle f,g\right\rangle _{\lambda
}=\int_{-1}^{1}f(x;q)g(x;q)(qx,-qx;q)_{\infty }d_{q}x+\lambda ({\mathscr D}%
_{q}f)(\alpha ;q)({\mathscr D}_{q}g)(\alpha ;q)  \label{qHermiteSobType}
\end{equation} of (\ref{qHermiteInnP}), where $\alpha \in \mathbb{R}\setminus \lbrack -1,1]$%
, $\lambda \in \mathbb{R}^{+}$ and ${\mathscr D}_{q}$ denotes the $q$%
-derivative operator, as will be defined below. In the sequel, we will
denote by $\left\{ S_{n}(x;q)\right\} _{n\geq 0}$ the sequence of monic
polynomials orthogonal with respect to (\ref{qHermiteSobType}).

The perturbation above, induced by discrete measures with a single mass-point and first-order $q$-derivatives, addresses a first step in the analysis of the Sobolev-type modifications involving $q$-derivatives. Thus, the
main goal of the paper will be to study this type of perturbations by providing the $q$-Hermite I-Sobolev-type polynomials $S_{n}(x;q)$, and to establish some interesting results concerning this family. We will assume here that the discrete mass $\lambda $ is located at a point $\alpha $ outside the support $[-1,1]$ of the $q$-discrete measure $(qx,-qx;q)_{\infty }d_{q}x$.

The structure of the manuscript is as follows. In Section \ref{S2-qCalculus} we summarize few preliminaries of the $q$-calculus which will be useful throughout the manuscript. In Section \ref{S3-qHermiteSobType} we obtain a structure relation for the $q$-derivatives of $S_{n}(x;q)$. In Section \ref{S4-Connectionf} we provide connection formulas between orthogonal polynomials of $q$-Hermite I and $q$-Hermite Sobolev-type families. Finally, in Section \ref{S5-Ladder} we obtain an explicit expression for the annihilation operator for the family ${S_{n}(x;q)}_n \geq 0$, in the same fashion as in \cite{HHL-PROMS10} and \cite{HHLS-M20}.



\section{Preliminaries results from $q$-calculus}

\label{S2-qCalculus}



In this section, we summarize few concepts and definitions from $q$-calculus, which will be needed in the sequel.

The $q$-number $[n]_{q}$, is defined by \cite{KLS2010} 
\begin{equation*}
\lbrack n]_{q}=%
\begin{cases}
\displaystyle0, & \mbox{ if }n=0, \\ 
&  \\ 
\displaystyle\frac{1-q^{n}}{1-q}=\sum_{0\leq k\leq n-1}q^{k}, & \mbox{ if }%
n\geq 1.%
\end{cases}%
\end{equation*}

A $q$-analogue of the factorial of $n$ is given by%
\begin{equation*}
\lbrack n]_{q}!=%
\begin{cases}
\displaystyle1, & \mbox{if }n=0, \\ 
&  \\ 
\displaystyle\lbrack n]_{q}[n-1]_{q}\cdots \lbrack 2]_{q}[1]_{q}, & 
\mbox{if
}n\geq 1.%
\end{cases}%
\end{equation*}%
Similarly, a $q$-analogue of the Pochhammer symbol, or shifted factorial, 
\cite{KLS2010} is defined by%
\begin{equation*}
\left( a;q\right) _{n}=%
\begin{cases}
\displaystyle1, & \mbox{if }n=0, \\ 
&  \\ 
\displaystyle\prod_{0\leq j\leq n-1}\left( 1-aq^{j}\right) , & \mbox{if }%
n\geq 1, \\ 
&  \\ 
\displaystyle(a;q)_{\infty }=\prod_{j\geq 0}(1-aq^{j}), & \mbox{if }n=\infty %
\mbox{ and }|a|<1.%
\end{cases}%
\end{equation*}%
Moreover, we will use the following notation 
\begin{equation*}
(a_{1},\ldots ,a_{r};q)_{k}=\prod_{1\leq j\leq r}(a_{j};q)_{k}.
\end{equation*}

The $q$-falling factorial (see \cite{arvesu2013first}) is defined by 
\begin{equation*}
\left[ s\right] _{q}^{(n)}=\frac{(q^{-s};q)_{n}}{(q-1)^{n}}q^{ns-\binom{n}{2}%
},\quad n\geq 1.
\end{equation*}%

The $q$-binomial coefficient is given by (see \cite{KLS2010})%
\begin{equation*}
\begin{bmatrix}
k \\ 
n%
\end{bmatrix}%
_{q}=\frac{(q;q)_{n}}{(q;q)_{k}(q;q)_{n-k}}=\frac{[n]_{q}!}{%
[k]_{q}![n-k]_{q}!}=%
\begin{bmatrix}
k \\ 
n-k%
\end{bmatrix}%
_{q},\quad k=0,1,\ldots ,n,
\end{equation*}%
where $n$ denotes a nonnegative integer.

The Jackson-Hahn-Cigler $q$-subtraction (see, for example \cite{ernst2007q} or \cite[Def. 6]{E-PEAS09}) and the references
given there)%
\begin{equation*}
\left( x\boxminus _{q}y\right) ^{n}=\prod_{0\leq j\leq n-1}\left(
x-yq^{j}\right) =x^{n}(y/x;q)_{n}=\sum_{0\leq k\leq n}%
\begin{bmatrix}
k \\ 
n%
\end{bmatrix}%
_{q}q^{\binom{k}{2}}(-y)^{k}x^{n-k}.
\end{equation*}%
The $q$-derivative \cite{KLS2010} or the Euler--Jackson $q$-difference
operator%
\begin{equation*}
({\mathscr D}_{q}f)(z)=%
\begin{cases}
\displaystyle\frac{f(qz)-f(z)}{(q-1)z}, & \text{if}\ z\neq 0,\ q\neq 1, \\ 
&  \\ 
f^{\prime }(z), & \text{if}\ z=0,\ q=1,%
\end{cases}%
\end{equation*}%
where ${\mathscr D}_{q}^{0}f=f$, ${\mathscr D}_{q}^{n}f={\mathscr D}_{q}({%
\mathscr D}_{q}^{n-1}f)$, with $n\geq 1$, and%
\begin{equation*}
\lim\limits_{q\rightarrow 1}{\mathscr D}_{q}f(z)=f^{\prime }(z).
\end{equation*}%
Furthermore, one has the following properties%
\begin{equation}
{\mathscr D}_{q}[f(\gamma z)]=\gamma ({\mathscr D}_{q}f)(\gamma z),\quad
\forall \,\,\gamma \in \mathbb{C},  \label{cadrule}
\end{equation}%
\begin{equation}
{\mathscr D}_{q}f(z)={\mathscr D}_{q^{-1}}f(qz)\Leftrightarrow {\mathscr D}%
_{q^{-1}}f(z)={\mathscr D}_{q}f(q^{-1}z),  \label{DqProp}
\end{equation}%
\begin{eqnarray}
{\mathscr D}_{q}[f(z)g(z)] &=&f(qz){\mathscr D}_{q}g(z)+g(z){\mathscr D}%
_{q}f(z)  \label{prodqD} \\
&=&f(z){\mathscr D}_{q}g(z)+g(qz){\mathscr D}_{q}f(z),  \notag
\end{eqnarray}%
and the following interesting property
\begin{equation}
{\mathscr D}_{q^{-1}}({\mathscr D}_{q}f)(z)=q{\mathscr D}_{q}({\mathscr D}%
_{q^{-1}}f)(z),  \label{DqProOK}
\end{equation}%
The $q$-Taylor formula (see \cite[Th. 6.3]{srivastava2011zeta}), with the
Cauchy remainder term, which is defined by%
\begin{equation*}
f(x)=\sum_{k=0}^{n}\frac{({\mathscr D}_{q}^{k}f)(a)}{[k]_{q}!}(x\boxminus
_{q}a)^{k}+\frac{1}{[n]_{q}!}\int_{a}^{x}({\mathscr D}_{q}^{n+1}f)(t)\cdot
(x\boxminus _{q}qt)^{n}d_{q}t.
\end{equation*}%
The Jackson $q$-integral is given by%
\begin{equation*}
\int_{0}^{z}f(x)d_{q}x=(1-q)z\sum_{k=0}^{\infty }q^{k}f(q^{k}z),
\end{equation*}%
which in a generic interval $[a,b]$ is given by%
\begin{equation*}
\int_{a}^{b}f(x)d_{q}x=\int_{0}^{b}f(x)d_{q}x-\int_{0}^{a}f(x)d_{q}x.
\end{equation*}



\section{$q$-Hermite I orthogonal polynomials}

\label{S3-qHermiteSobType}



After the above $q$-calculus introduction, we continue by giving several
aspects and properties of the $q$-Hermite I polynomials $\{H_{n}(x;q)\}_{n%
\geq 0}$.\\

The monic $q$-Hermite I polynomials can be given by means of their generating function:%
\begin{equation*}
\sum_{n=0}^{\infty }{H_{n}(x;q)\dfrac{t^{n}}{(q;q)_{n}}}=\prod_{n=0}^{\infty
}{\dfrac{1}{1-2xtq^{n}+t^{2}q^{2n}}.}
\end{equation*}


We also have the following classic and well-known results in the literature for this family of polynomials:
\begin{proposition}
\label{S1-Proposition11} Let $\{H_{n}(x;q)\}_{n\geq 0}$ be the sequence of $%
q $-Hermite I polynomials of degree $n$. Then next statements hold.

\begin{enumerate}
\item The recurrence relation \cite{KLS2010} 
\begin{equation}
xH_{n}(x;q) =H_{n+1}(x;q)+\gamma_n H_{n-1}(x;q),  \label{ReR}
\end{equation}
with initial conditions $H_{-1}(x;q)=0$ and $H_{0}(x;q)=1$. Here, $%
\gamma_n=q^{n-1}(1-q^{n})$.

\item Squared norm \cite{KLS2010}. For every $n\in \mathbb{N}$, 
\begin{equation*}
||H_{n}||^{2}=(1-q)(q;q)_n(q,-1,-q;q)_{\infty}q^{\binom{n}{2}}.
\end{equation*}

\item Forward shift operator \cite{KLS2010} 
\begin{equation}
{\mathscr D}_{q}^{k}H_{n}(x;q) =[n]_{q}^{(k)}H_{n-k}(x;q) ,  \label{FSop}
\end{equation}
where 
\begin{equation*}
\left[ n\right] _{q}^{(k)}=\frac{(q^{-n};q)_{k}}{(q-1)^{k}}q^{kn-\binom{k}{2}
},
\end{equation*}
denote the $q$-falling factorial \cite{arvesu2013first}. Observe that $\left[ n\right] _{q}^{(1)}=\left[ n\right] _{q}$.

\item Second-order $q$-difference equation \cite{KLS2010} 
\begin{equation*}
\sigma(x){\mathscr D}_q{\mathscr D}_{q^{-1}}H_{n}(x;q)+\tau(x){\mathscr D}%
_{q}H_{n}(x;q)+\lambda_{n,q}H_{n}(x;q)=0,
\end{equation*}
where $\sigma(x)=x^2-1$, $\tau(x)=(1-q)^{-1}x$ and $%
\lambda_{n,q}=[n]_q([1-n]_q\sigma^{\prime \prime }/2-\tau^{\prime })$.
\end{enumerate}
\end{proposition}


A Christoffel-Darboux formula for this family can also be established:
\begin{proposition}[Christoffel-Darboux formula]
\label{S1-Proposition12} Let $\{H_{n}(x;q)\}_{n\geq 0}$ be the sequence of $%
q $-Hermite I polynomials. If we denote the $n$-th reproducing kernel by 
\begin{equation*}
K_{n,q}(x,y)=\sum_{k=0}^{n}\frac{H_{k}(x;q)H_{k}(y;q)}{||H_{k}||^{2}}.
\end{equation*}%
Then, for all $n\in \mathbb{N}$, it holds that%
\begin{equation}
K_{n,q}(x,y)=\frac{H_{n+1}(x;q)H_{n}(y;q)-H_{n+1}(y;q)H_{n}(x;q)}{\left(
x-y\right) ||H_{n}||^{2}}.  \label{CDarb}
\end{equation}
\end{proposition}



Concerning the partial $q$-derivatives of $K_{n,q}(x,y)$, we use the
following notation
\begin{eqnarray*}
K_{n,q}^{(i,j)}(x,y) &=&{\mathscr D}^{j}_{q,y}({\mathscr D}^{i}_{q,x}K_{n,q}(x,y)) \\
&=&\sum_{k=0}^{n}\frac{{\mathscr D}_{q}^{i}H_{k}(x;q){\mathscr D}%
_{q}^{j}H_{k}(y;q)}{||H_{k}||^{2}}.
\end{eqnarray*}

In order to establish connection formulas between the polynomials of $\{S_{n}(x;q)\}_{n\geq 0}$ and the polynomials of $\{H_{n}(x;q)\}_{n\geq 0}$, we first need to relate the first-order partial derivatives of the kernel polynomials of $\{H_{n}(x;q)\}_{n\geq 0}$ to the polynomials of this family. The next result is, in fact, a connection formula between the partial $q$-derivative with respect to one variable of the kernel polynomials and the polynomials of $\{H_{n}(x;q)\}_{n\geq 0}$. The proof is quite involved, and follows from the techniques recently presented in \cite{HHLS-M20} .



\begin{proposition}
\label{S1-LemmaKernel0j} Let $\{H_{n}(x;q)\}_{n\geq 0}$ be the sequence of $%
q $-Hermite I polynomials of degree $n$. Then following statements hold, for
all $n\in \mathbb{N}$, 
\begin{equation}
K_{n-1,q}^{(0,1)}(x,y)={\mathcal{A}}_{n}^{(1)}(x,y)H_{n}(x;q)+{\mathcal{B}}%
_{n}^{(1)}(x,y)H_{n-1}(x;q), \label{Kernel0j}
\end{equation}
where 
\begin{eqnarray*}
{\mathcal{A}}_{n}^{(1)}(x,y)=\frac{\left[ 1\right] _{q}!}{||H_{n-1}||^{2}
\left( x\boxminus _{q}y\right) ^{2}}\sum_{k=0}^{1}\frac{{\mathscr D}%
_{q}^{k}H_{n-1}(y;q)}{\left[ k\right] _{q}!}(x\boxminus _{q}y)^{k}= \\
\frac{1}{||H_{n-1}||^{2}
\left( x\boxminus _{q}y\right) ^{2}}\left(H_{n-1}(y;q)+{\mathscr D}%
_{q}H_{n-1}(y;q)(x\boxminus _{q}y)\right)
,
\end{eqnarray*}
and 
\begin{eqnarray*}
{\mathcal{B}}_{n}^{(1)}(x,y)=-\frac{\left[ j\right] _{1}!}{%
||H_{n-1}||^{2}\left( x\boxminus _{q}y\right) ^{2}}\sum_{k=0}^{1}\frac{{%
\mathscr D}_{q}^{k}H_{n}(y;q)}{\left[ k\right] _{q}!}(x\boxminus _{q}y)^{k}= \\
-\frac{1}{%
||H_{n-1}||^{2}\left( x\boxminus _{q}y\right) ^{2}}\left( x\boxminus _{q}y\right) ^{2}\left(H_{n}(y;q)+{\mathscr D}%
_{q}H_{n}(y;q)(x\boxminus _{q}y)\right).
\end{eqnarray*}
\end{proposition}

Furthermore, we can state a connection formula between the partial $q$-derivatives with respect to both variables $x,y$ of the kernel polynomials, and the polynomials of $\{H_{n}(x;q)\}_{n\geq 0}$:
\begin{proposition}
\label{S1-LemmaKerneli2} Let $\{H_{n}(x;q)\}_{n\geq 0}$ be the sequence of $%
q $-Hermite I polynomials of degree $n$. Then following statement holds, for
all $n\in \mathbb{N}$, 
\begin{eqnarray}
K_{n-1,q}^{(1,1)}(x,y) &=&{\mathcal{C}}_{1,n}(x,y)H_{n}(x;q)+{\mathcal{D}}%
_{1,n}(x,y)H_{n-1}(x;q), \label{kernel1j}
\end{eqnarray}%
where%
\begin{equation*}
{\mathcal{C}}_{1,n}(x,y)={\mathscr D}_{q}{\mathcal{A}}%
_{n}^{(1)}(x,y)-[n-1]_{q}\gamma _{n-1}^{-1}{\mathcal{B}}_{n}^{(1)}(qx,y)
\end{equation*}%
and
\begin{equation*}
{\mathcal D}_{1,n}(x,y)=[n]_{q}{\mathcal{A}}_{n}^{(1)}(qx,y)+[n-1]_{q}\gamma
_{n-1}^{-1}x{\mathcal{B}}_{n}^{(1)}(qx,y)+{\mathscr D}_{q}{\mathcal{B}}%
_{n}^{(1)}(x,y).
\end{equation*}%
\end{proposition}



\begin{proof}
Applying the $q$-derivative operator ${\mathscr D}_{q}$ to (\ref%
{Kernel0j}), together with the property (\ref{prodqD}), we have%
\begin{eqnarray*}
K_{n-1,q}^{(1,1)}(x,y) &=&{\mathcal{A}}_{n}^{(1)}(qx){\mathscr D}%
_{q}H_{n}(x;q)+H_{n}(x;q){\mathscr D}_{q}{\mathcal{A}}_{n}^{(1)}(x) \\
&&+{\mathcal{B}}_{n}^{(1)}(qx){\mathscr D}_{q}H_{n-1}(x;q)+H_{n-1}(x;q){%
\mathscr D}_{q}{\mathcal{B}}_{n}^{(1)}(x).
\end{eqnarray*}%
Using (\ref{FSop}) and (\ref{ReR}) we easily deduce (\ref{kernel1j}), and the proof is complete.
\end{proof}



\section{Connection formulas}

\label{S4-Connectionf}

\label{S4-ConnForm}



In the Introduction, we have presented the $q$-Hermite I-Sobolev type orthogonal
polynomials $\{S_{n}(x;q)\}_{n\geq 0}$, which are orthogonal with respect to
Sobolev-type inner product%
\begin{equation*}
\left\langle f,g\right\rangle _{\lambda
}=\int_{-1}^{1}f(x;q)g(x;q)(qx,-qx;q)_{\infty }d_{q}x+\lambda ({\mathscr D}%
_{q}f)(\alpha ;q)({\mathscr D}_{q}g)(\alpha ;q),
\end{equation*}%
where $\alpha \in \mathbb{R}\setminus \lbrack -1,1]$, $\lambda \in \mathbb{R}%
^{+}$ and $j\in \mathbb{N}$. Now, in this Section, we will express these $q$-Hermite I-Sobolev
type orthogonal polynomials $\{S_{n}(x;q)\}_{n\geq 0}$ in terms of the $q$%
-Hermite polynomials $\{H_{n}(x;q)\}_{n\geq 0}$, the kernel polynomials and
their corresponding derivatives. The main idea behind the identities
presented here, is that of expressing the new polynomials in
terms of another with well-established properties, to infer, from these,
those of the first one. As a consequence, we will obtain connection formulas between the polynomials of $\{S_{n}(x;q)\}_{n\geq 0}$ and $\{H_{n}(x;q)\}_{n\geq 0}$, as well as between the first-order $q$-derivatives of the polynomials of this new family and the polynomials of the old one.



We begin by relating the $q$-Hermite I-Sobolev type polynomials and their $q$-derivatives to the $q$-Hermite polynomials and the kernel polynomials:
\begin{proposition}
Let $\{S_{n}(x;q)\}_{n\geq 0}$ be the sequence of $q$-Hermite I-Sobolev type
orthogonal polynomials of degree $n$. Then, following statements hold, 
\begin{equation}
S_{n}(x;q)=H_{n}(x;q)-\lambda \frac{\lbrack n]_{q}^{(1)}H_{n-1}(\alpha ;q)}%
{1+\lambda K_{n-1,q}^{(1,1)}(\alpha ,\alpha )}K_{n-1,q}^{(0,1)}(x,\alpha ).
\label{ConxF1}
\end{equation}
\end{proposition}



\begin{proof}
Taking into account the Fourier expansion%
\begin{equation*}
S_{n}(x;q)=H_{n}(x;q)+\sum_{0\leq k\leq n-1}a_{n,k}H_{k}(x;q).
\end{equation*}%
Next, from (\ref{DqProp}) and considering the orthogonality properties of $%
H_{n}(x;q)$, the coefficients in the previous expansion are given by 
\begin{equation*}
a_{n,k}=-\frac{\lambda {\mathscr D}_{q}S_{n}(\alpha ;q){\mathscr D}%
_{q}H_{k}(\alpha ;q)}{||H_{k}||^{2}},\quad 0\leq k\leq n-1.
\end{equation*}%
Thus%
\begin{equation*}
S_{n}(x;q)=H_{n}(x;q)-\lambda {\mathscr D}_{q}S_{n}(\alpha
;q)K_{n-1,q}^{(0,1)}(x,\alpha ).
\end{equation*}%
Applying the operator ${\mathscr D}_{q}$ to the previous equation we get%
\begin{equation*}
{\mathscr D}_{q}S_{n}(x;q)=[n]_{q}^{(1)}H_{n-1}(x;q)-\lambda {\mathscr D}%
_{q}S_{n}(\alpha ;q)K_{n-1,q}^{(1,1)}(x,\alpha ).
\end{equation*}%
After some manipulations, we deduce%
\begin{equation*}
{\mathscr D}_{q}S_{n}(\alpha ;q)=\frac{[n]_{q}^{(1)}H_{n-1}(\alpha ;q)}{%
1+\lambda K_{n-1,q}^{(1,1)}(\alpha ,\alpha )}.
\end{equation*}%
Therefore, we obtain (\ref{ConxF1}).
\end{proof}



As a consequence, we have the following result:

\begin{corollary}
\label{DqDq2HS} Let $\{S_{n}(x;q)\}_{n\geq 0}$ be the sequence of $q$%
-Hermite I-Sobolev type orthogonal polynomials of degree $n$. Then,
following statement holds, 
\begin{equation*}
{\mathscr D}_{q}S_{n}(x;q)=[n]_{q}^{(1)}H_{n-1}(x;q)-\lambda \frac{\lbrack
n]_{q}^{(1)}H_{n-1}(\alpha ;q)}{1+\lambda K_{n-1,q}^{(1,1)}(\alpha ,\alpha )}%
K_{n-1,q}^{(1,1)}(x,\alpha ).
\end{equation*}%
\end{corollary}


Now, we are in a position to state connection formulas between the orthogonal polynomials of the $q$-Hermite I and $q$-Hermite I-Sobolev type families. These lemmas provide the required results:

\begin{lemma}
Let $\{S_{n}(x;q)\}_{n\geq 0}$ be the sequence of $q$-Hermite I-Sobolev type
orthogonal polynomials of degree $n$. Then, we have 
\begin{equation}
S_{n}(x;q)={\mathcal{E}}_{1,n}(x)H_{n}(x;q)+F_{1,n}(x)H_{n-1}(x;q),
\label{ConexF_I}
\end{equation}%
where 
\begin{equation*}
{\mathcal{E}}_{1,n}(x)=1-\lambda \frac{\lbrack n]_{q}^{(1)}H_{n-1}(\alpha ;q)%
}{1+\lambda K_{n-1,q}^{(1,1)}(\alpha ,\alpha )}{\mathcal{A}}%
_{n}^{(1)}(x,\alpha ),
\end{equation*}
and 
\begin{equation*}
{\mathcal{F}}_{1,n}(x)=-\lambda \frac{\lbrack n]_{q}^{(1)}H_{n-1}(\alpha ;q)%
}{1+\lambda K_{n-1,q}^{(1,1)}(\alpha ,\alpha )}{\mathcal{B}}%
_{n}^{(1)}(x,\alpha ).
\end{equation*}
\end{lemma}



\begin{proof}
From (\ref{ConxF1}) and the Proposition \ref{S1-LemmaKernel0j}, the lemma holds.
\end{proof}



On the other hand, from previous Lemma and recurrence relation (\ref{ReR})
we obtain the following result 
\begin{equation}
S_{n-1}(x;q)={\mathcal{E}}_{2,n}(x)H_{n}(x;q)+{\mathcal{F}}%
_{2,n}(x)H_{n-1}(x;q),  \label{ConexF_II}
\end{equation}%
where 
\begin{equation*}
{\mathcal{E}}_{2,n}(x)=-\frac{{\mathcal{F}}_{1,n-1}(x)}{\gamma _{n-1}},
\end{equation*}%
and 
\begin{equation*}
{\mathcal{F}}_{2,n}(x)={\mathcal{E}}_{1,n-1}(x)-x{\mathcal{E}}_{2,n}(x).
\end{equation*}



\begin{lemma}
\label{detxi} Let $\{S_{n}(x;q)\}_{n\geq 0}$ be the sequence of $q$-Hermite
I-Sobolev type orthogonal polynomials of degree $n$. Then, next
statements hold, 
\begin{equation}
\Xi _{1,n}(x)H_{n}(x;q)=%
\begin{vmatrix}
S_{n}(x;q) & S_{n-1}(x;q) \\ 
{\mathcal{F}}_{1,n}(x) & {\mathcal{F}}_{2,n}(x)%
\end{vmatrix}%
,  \label{ConexF_III}
\end{equation}%
and 
\begin{equation}
\Xi _{1,n}(x)H_{n-1}(x;q)=-%
\begin{vmatrix}
S_{n}(x;q) & S_{n-1}(x;q) \\ 
{\mathcal{E}}_{1,n}(x) & {\mathcal{E}}_{2,n}(x)%
\end{vmatrix}%
,  \label{ConexF_IV}
\end{equation}%
where 
\begin{equation*}
\Xi _{1,n}(x)=%
\begin{vmatrix}
{\mathcal{E}}_{1,n}(x) & {\mathcal{E}}_{2,n}(x) \\ 
{\mathcal{F}}_{1,n}(x) & {\mathcal{F}}_{2,n}(x)%
\end{vmatrix}%
.
\end{equation*}
\end{lemma}



\begin{proof}
Multiplying (\ref{ConexF_I}) by ${\mathcal{F}}_{2,n}(x)$ and (\ref{ConexF_II}) by $-{\mathcal{F}}_{1,n}(x)$, adding and
simplifying the resulting equations, we deduce (\ref{ConexF_III}). In
addition, we can now proceed analogously to get (\ref{ConexF_IV}).
\end{proof}


We end this section stating another connection formula between the first-order $q$-derivatives of the $q$-Hermite I-Sobolev type orthogonal polynomials and the $q$-Hermite I polynomials:

\begin{lemma}
\label{DqHS} Let $\{S_{n}(x;q)\}_{n\geq 0}$ be the sequence of $q$-Hermite
I-Sobolev type orthogonal polynomials of degree $n$. Then, following
statements hold, 
\begin{equation}
{\mathscr D}_{q}S_{n}(x;q)={\mathcal{E}}_{3,n}(x)H_{n}(x;q)+{\mathcal{F}}%
_{3,n}(x)H_{n-1}(x;q),  \label{DqHsE3F3}
\end{equation}%
where
\begin{equation*}
{\mathcal{E}}_{3,n}(x)=-\lambda \frac{\lbrack
n]_{q}^{(1)}H_{n-1}(\alpha ;q)%
}{1+\lambda K_{n-1,q}^{(1,1)}(\alpha ,\alpha )}{\mathcal{C}}_{1,n}(x,\alpha
),
\end{equation*}%
and 
\begin{equation*}
{\mathcal{F}}_{3,n}(x)=[n]_{q}^{(1)}-\lambda \frac{\lbrack
n]_{q}^{(1)}H_{n-1}(\alpha ;q)}{1+\lambda K_{n-1,q}^{(1,1)}(\alpha ,\alpha )}%
{\mathcal{D}}_{1,n}(x,\alpha ).
\end{equation*}
\end{lemma}



\begin{proof}

The statements easily follow by replacing the expansion of $K_{n-1,q}^{(1,1)}(x,y)$ as a linear combination of $H_{n}(x;q)$ and $H_{n-1}(x;q)$ (formula \ref{kernel1j}, Proposition \ref{S1-LemmaKerneli2}) in the expression of ${\mathscr D}_{q}S_{n}(x;q)$ (Corollary \ref{DqDq2HS}):

\begin{equation*}
{\mathscr D}_{q}S_{n}(x;q)=
\end{equation*}
\begin{equation*}
[n]_{q}^{(1)}H_{n-1}(x;q)-\lambda \frac{\lbrack n]_{q}^{(1)}H_{n-1}(\alpha ;q)}{1+\lambda K_{n-1,q}^{(1,1)}(\alpha ,\alpha )}%
\left( {\mathcal{C}}_{1,n}(x,y)H_{n}(x;q)+{\mathcal{D}}%
_{1,n}(x,y)H_{n-1}(x;q) \right)
\end{equation*}

and rearranging the terms keeping the linear combination of the two consecutive $q$-Hermite I polynomials of degrees $n$ and $n-1$:

\begin{equation*}
{\mathscr D}_{q}S_{n}(x;q)=
\end{equation*}
\begin{equation*}
[n]_{q}^{(1)}H_{n-1}(x;q)-\lambda \frac{\lbrack
n]_{q}^{(1)}H_{n-1}(\alpha ;q)}{1+\lambda K_{n-1,q}^{(1,1)}(\alpha ,\alpha )}%
\left( {\mathcal{C}}_{1,n}(x,y)H_{n}(x;q)+{\mathcal{D}}%
_{1,n}(x,y)H_{n-1}(x;q) \right)= 
\end{equation*}
\begin{equation*}
[n]_{q}^{(1)}H_{n-1}(x;q)-\lambda \frac{\lbrack
n]_{q}^{(1)}H_{n-1}(\alpha ;q)}{1+\lambda K_{n-1,q}^{(1,1)}(\alpha ,\alpha )}%
{\mathcal{C}}_{1,n}(x,y)H_{n}(x;q)
\end{equation*}
\begin{equation*}
-\lambda \frac{\lbrack
n]_{q}^{(1)}H_{n-1}(\alpha ;q)}{1+\lambda K_{n-1,q}^{(1,1)}(\alpha ,\alpha )}%
{\mathcal{D}}%
_{1,n}(x,y)H_{n-1}(x;q)= 
\end{equation*}
\begin{equation*}
\left( -\lambda \frac{\lbrack
n]_{q}^{(1)}H_{n-1}(\alpha ;q)}{1+\lambda K_{n-1,q}^{(1,1)}(\alpha ,\alpha )}%
{\mathcal{C}}_{1,n}(x,y) \right) H_{n}(x;q)
\end{equation*}
\begin{equation*}
+ \left( [n]_{q}^{(1)}-\lambda \frac{\lbrack
n]_{q}^{(1)}H_{n-1}(\alpha ;q)}{1+\lambda K_{n-1,q}^{(1,1)}(\alpha ,\alpha )}%
{\mathcal{D}}%
_{1,n}(x,y) \right) H_{n-1}(x;q).
\end{equation*}

\end{proof}



\section{Annihilation operator}

\label{S5-Ladder}



In this section, we finally provide an expression of the annihilation operator for this new family of $q$-Hermite I-Sobolev type polynomials:

\begin{proposition}[Annihilation operator]
\label{STHST} The annihilation operator of the $q$-Hermite I-Sobolev type
orthogonal polynomials $S_{n}(x;q)$, has the following explicit expression%
\[
\mathfrak{a}_{n}\,S_{n}(x;q)=S_{n-1}(x;q),
\]%
where%
\[
\mathfrak{a}_{n}:=\frac{1}{{\mathcal{F}}_{4,n}(x)}\left( \Xi _{1,n}(x){%
\mathscr D}_{q}-{\mathcal{E}}_{4,n}(x)I\right) ,
\]%
being $I$ the identity operator. One also has%
\[
{\mathcal{E}}_{4,n}(x)=-\begin{vmatrix} {\mathcal{E}}_{2,n}(x) &
{\mathcal{E}}_{3,n}(x) \\ {\mathcal{F}}_{2,n}(x) &
{\mathcal{F}}_{3,n}(x)\end{vmatrix},
\]%
and 
\[
{\mathcal{F}}_{4,n}(x)=\begin{vmatrix} {\mathcal{E}}_{1,n}(x) &
{\mathcal{E}}_{3,n}(x) \\ {\mathcal{F}}_{1,n}(x) &
{\mathcal{F}}_{3,n}(x)\end{vmatrix}.
\]
\end{proposition}



\begin{proof}
Using the Lemmas \ref{detxi} and \ref{DqHS}, respectively, we get%
\[
\begin{vmatrix} S_{n}(x;q) & S_{n-1}(x;q) \\ {\mathcal{F}}_{1,n}(x) &
{\mathcal{F}}_{2,n}(x)\end{vmatrix}{\mathcal{E}}_{3,n}(x)-\begin{vmatrix}
S_{n}(x;q) & S_{n-1}(x;q) \\ {\mathcal{E}}_{1,n}(x) &
{\mathcal{E}}_{2,n}(x)\end{vmatrix}{\mathcal{F}}_{3,n}(x)
\]%
\begin{eqnarray*}
&=&{\mathcal{E}}_{3,n}(x){\mathcal{F}}_{2,n}(x)S_{n}(x;q)-{\mathcal{E}}%
_{3,n}(x){\mathcal{F}}_{1,n}(x)S_{n-1}(x;q) \\
&&\qquad -{\mathcal{E}}_{2,n}(x){\mathcal{F}}_{3,n}(x)S_{n}(x;q)+{\mathcal{E}%
}_{1,n}(x){\mathcal{F}}_{3,n}(x)S_{n-1}(x;q)
\end{eqnarray*}%
\[
=-\begin{vmatrix} {\mathcal{E}}_{2,n}(x) & {\mathcal{E}}_{3,n}(x) \\
{\mathcal{F}}_{2,n}(x) & {\mathcal{F}}_{3,n}(x)\end{vmatrix}S_{n}(x;q)+%
\begin{vmatrix} {\mathcal{E}}_{1,n}(x) & {\mathcal{E}}_{3,n}(x) \\
{\mathcal{F}}_{1,n}(x) & {\mathcal{F}}_{3,n}(x)\end{vmatrix}S_{n-1}(x;q).
\]%
The proof is complete after reordering terms in the above expression.
\end{proof}



\section{Conclusions}

\label{S6-Conclusions}



In this paper, a new family of $q$-Hermite I-Sobolev type polynomials, orthogonal with respect to the $q$-discrete measure of the $q$-Hermite I polynomials modified by means of discrete measures with a single mass-point and first-order $q$-derivatives, has been presented. Interesting and intrincate results for this sequence of polynomials, such as several connection formulas and an explicit expression for its annihilation operator, have also been obtained.

A pending question is the expression for the corresponding creation operator, as well as a second order $q$-difference equation satisfied by this $q$-Hermite I Sobolev-type polynomials. Thus, the results presented in this paper can be the starting point to study the modified $q$-Hermite polynomials by adding more mass-points at the discrete part of the measure to later include higher-order $q$-derivatives. This study is not easy, since it involves a greater number of calculations as well as a greater complexity in the expressions, so new strategies are required.

Another important line of future research would deal with the analysis of the behavior of the zeros of these families and their possible physical interpretations.



\section*{Acknowledgments}



The work of the first author (C.H.) was funded by Direcci\'on General de Investigaci\'on e Innovaci\'on, Consejer\'ia de Educaci\'on e Investigaci\'on of the Comunidad de Madrid (Spain), and Universidad de Alcal\'a, under grant CM/JIN/2019-010, Proyectos de I+D para J\'{o}venes Investigadores de la Universidad de Alcal\'a 2019.




\begin{thebibliography}{99}


\bibitem{Adig} S. Alwhishi, R. S. Adig\"uzel, M. Turan, On the Orthogonality of the $q$-Derivatives of the Discrete $q$-Hermite I Polynomials, \emph{Emerging Applications of Differential Equations and Game Theory}, IGI Global, 135--162 (2020). 


\bibitem{arvesu2013first} J. Arves\'{u} and A. Soria--Lorente, First-order non-homogeneous $q$-difference equation for Stieltjes function
characterizing $q$-orthogonal polynomials, \emph{J. Difference Equ. Appl.} \textbf{19} (5), 814--838 (2013).

\bibitem{BI-CJM96} C. Berg, and M. E. H. Ismail, Q-Hermite polynomials and classical orthogonal polynomials, \emph{Can. J. Math.}, \textbf{48} (1), 43--63 (1996).

\bibitem{C1956} L. Carlitz, Some polynomials related to theta functions, \emph{Ann. Mat. Pura Appl.} \textbf{41} (4), 359--373 (1956).

\bibitem{C1957} L. Carlitz, Some polynomials related to Theta functions, \emph{Duke Math. J.} \textbf{24}, 521--527 (1957).

\bibitem{C1972} L. Carlitz, Generating functions for certain $q$-orthogonal polynomials, \emph{Collect. Math.} \textbf{23}, 91--104 (1972).

\bibitem{HHL-PROMS10} C. Hermoso, E. J. Huertas, and A. Lastra, Determinantal form for ladder operators in a problem concerning a convex linear combination of discrete and continuous measures, \emph{G. Filipuk et al. (Eds.): FASdiff 2017 - Springer Proceedings in Mathematics and Statistics (PROMS)}, \textbf{256}, 263--274 (2018).

\bibitem{HHLS-M20} C. Hermoso, E. J. Huertas, A. Lastra, and Anier Soria-Lorente, On Second Order $q$-Difference Equations Satisfied by Al-Salam-Carlitz I-Sobolev Type Polynomials of Higher Order, \emph{Mathematics}, \textbf{8} (8), 1300 (2020).

\bibitem{ernst2007q} T. Ernst, $q$-Complex Numbers, A natural consequence of umbral calculus. \emph{Uppsala University Department of Mathematics}, Report \textbf{44} (2007).

\bibitem{E-PEAS09} T. Ernst, $q$-Calculus as operational algebra, \emph{Proceedings of the Estonian Academy of Sciences}, \textbf{58} (2), 73--97 (2009).

\bibitem{Galin} G. Filipuk, J. F. Ma\~nas-Ma\~nas and J. J. Moreno-Balc\'azar, \emph{Ladders operators for general discrete Sobolev orthogonal polynomials}. ArXiv:2006.14391v1 [math.CA] (2020).

\bibitem{H1949} W. Hahn, Über Orthogonalpolynome, die $q$-Differenzengleichungen genügen. \emph{Math. Nachr.} \textbf{2}, 4--34
(1949).

\bibitem{KLS2010} R. Koekoek, P. A. Lesky, and R. F. Swarttouw, Hypergeometric orthogonal polynomials and their $q$-analogues. Springer
Science \& Business Media (2010).

\bibitem{R1893-1} L. J. Rogers, On the expansion of certain infinite products, \emph{Proc. London Math. Soc.}, \textbf{24}, 337--352 (1893).

\bibitem{R1894-2} L. J. Rogers, Second memoir on the expansion of certain infinite products, \emph{Proc. London Math. Soc.}, \textbf{25}, 318--343 (1894).

\bibitem{R1895-3} L. J. Rogers, Third memoir on the expansion of certain infinite products, \emph{Proc. London Math. Soc.}, \textbf{26}, 15--32 (1895).

\bibitem{srivastava2011zeta} H. M. Srivastava and J. Choi, \emph{Zeta and $q$-Zeta functions and associated series and integrals}, Elsevier (2001).

\bibitem{Sz1926} G. Szeg\H{o}, Ein Beitrag zur Theorie der Thetafunktionen, \emph{Sitz. Preuss. Akad. Wiss. Phys. Math. KI.}, \textbf{XIX}: 242--252 (1926).

\end{thebibliography}
\end{document}